\documentclass[10pt,oneside,reqno]{amsart}

\RequirePackage{fix-cm} 
\usepackage[english]{babel}
\usepackage{amssymb,amsmath,geometry,amsthm,graphicx,enumerate}
\usepackage[dvipsnames]{xcolor}
\usepackage{lmodern}

\usepackage[T1]{fontenc}
\usepackage{hyperref}
\hypersetup{
 colorlinks,
 citecolor=Maroon,
 linkcolor=Maroon,
 urlcolor=Maroon}

\newtheorem{definition}{Definition}
\newtheorem{lemma}{Lemma}
\newtheorem{proposition}{Proposition}

\newtheorem{theorem}{Theorem}

\theoremstyle{remark}
\newtheorem{remark}{\sc  Remark\rm}[section]
\newtheorem{notation}{\sc  Notation\rm}[section]

\newcommand{\assign}{:=}
\newcommand{\mathd}{\mathrm{d}}
\newcommand{\nobracket}{}
\newcommand{\of}{:}
\newcommand{\tmabbr}[1]{#1}
\newcommand{\tmcolor}[2]{{\color{#1}{#2}}}
\newcommand{\tmem}[1]{{\em #1\/}}
\newcommand{\tmmathbf}[1]{\ensuremath{\boldsymbol{#1}}}
\newcommand{\tmop}[1]{\ensuremath{\operatorname{#1}}}
\newcommand{\tmstrong}[1]{\textbf{#1}}
\newcommand{\tmtextit}[1]{{\itshape{#1}}}
\newcommand{\tmtextsc}[1]{{\scshape{#1}}}

\newcounter{marnote}

\newcommand{\Fk}{\mathcal{F}_{\kappa}}
\newcommand{\Gk}{\mathcal{G}_{\kappa}}
\newcommand{\myJ}{J}
\newcommand{\ns}{n_{\ast}}
\newcommand{\vhs}{\mathbf{{\ell}}'_2 (J)}
\newcommand{\ltwoJ}{\mathbf{{\ell}}_2 (J)}
\newcommand{\tmsep}{, }
\newcommand{\lapl}{\Delta}
\newcommand{\grad}{\nabla}

\newcommand{\RR}{\mathbb{R}}
\newcommand{\ZZ}{\mathbb{Z}}
\newcommand{\NN}{\mathbb{N}}

\newcommand{\Stwo}{\mathbb{S}}

      plus .9\fontdimen3\font
     minus .9\fontdimen4 

\DeclareRobustCommand{\subtitle}[1]{\\#1}

\begin{document}

\title{On a sharp  Poincar{\'e}-type inequality on the  2-sphere  \subtitle{and its application in micromagnetics}}

\author{Giovanni Di Fratta}
\address{{Giovanni Di Fratta,\,\small{  Institute for Analysis \\
and Scientific Computing, TU Wien\\
Wiedner Hauptstra{\ss}e 8-10 \\
1040 Wien, Austria.}}}

\author{Valeriy Slastikov}
\address{{Valeriy Slastikov,\,\small{School of Mathematics \\
University of Bristol \\
University Walk, Bristol \\
BS8 1TW, United Kingdom.}}}

\author{Arghir Zarnescu}
\address{Arghir Zarnescu, \,\small{IKERBASQUE, Basque Foundation for Science, Maria Diaz de Haro 3,
48013, Bilbao, Bizkaia, Spain.}}
\address{\small{BCAM,  Basque  Center  for  Applied  Mathematics,  Mazarredo  14,  E48009  Bilbao,  Bizkaia,  Spain.}}
 \address {\small{``Simion Stoilow" Institute of the Romanian Academy, 21 Calea Grivi\c{t}ei, 010702 Bucharest, Romania.} }

\begin{abstract}
  
  The main aim of
  this note is to prove a sharp Poincar{\'e}-type inequality for vector-valued functions on $\Stwo^2$, that naturally emerges in the context of 
  micromagnetics of spherical thin films. 
  
\vspace{6pt}
\noindent \tmstrong{Keywords.} Poincar{\'e} inequality{\tmsep}{vector spherical harmonics}{\tmsep}{magnetic skyrmions}

\end{abstract}

\maketitle
\renewcommand{\subtitle}[1]{}

\section{Introduction}
The Poincar{\'e}-type inequalities are a crucial tool in analysis, as they provide a relation between the norms of a function and its gradient. As such they are deeply relevant in analytic models appearing in geometry, physics and biology. Such models often exhibit different qualitative behaviours for various ranges of  parameters and therefore sharply estimating the Poincar\'e constant is fundamental for a proper understanding of a model. 

 The Poincar\'e-type inequalities always involve some constraints on the target of the function in order to eliminate the constants, which are not seen by the gradient part. The most commonly used ones, for scalar-valued functions, involve either local restrictions (zero values on the boundary of the domain) or non-local ones (zero mean). The optimal constant strongly depends on the type of  constraint imposed and provides a piece of significant geometric information about the problem under consideration~{\cite{beckner1993sharp,zhu2004extremal,hebey1996sobolev}}.

  There exists an enormous body of literature about  Poincar{\'e}-type inequalities for {\it scalar-valued} functions but virtually nothing about {\it vector-valued} ones despite  their use in many physical contexts. The last four decades have witnessed an extraordinary interest in manifold-valued function spaces but Poincar\'e inequalities  naturally relevant in this context have not been explored much. The various constraints on the range of the vector-valued function, motivated by physical or geometrical considerations reduce the degrees of freedom allowed on the function and generate natural questions concerning the optimal constants. Such questions require special approaches, going beyond what is available in the scalar case.

We are interested in proving a sharp Poincar{\'e}-type inequality for vector-valued functions on the 2-sphere $\Stwo^2 \assign
\left\{ \xi \in \RR^3 \of | \xi | = 1 \right\}$ and using this result to obtain  non-trivial information about magnetization behaviour inside thin spherical shells.  Topological magnetic structures arising in non-flat geometries attract a lot of interest due to their potential in the application to magnetic devices \cite{streubel16}. Thin spherical shells are one of the simplest examples where an interplay between topology, geometry and curvature of the underlying space results in non-trivial magnetic structures \cite{sloika17}.

The magnetization distribution $\tmmathbf{u} \in
H^1( \Stwo^2, \Stwo^2)$ in thin spherical shells can be found by minimizing the following reduced micromagnetic energy \cite{di2016dimension,kravchuk2016topologically}
\begin{equation}
  \Fk (\tmmathbf{u}) = \int_{\Stwo^2} \left| \grad^{\ast}_{\xi}
  \tmmathbf{u} (\xi) \right|^2 \mathd \xi + \kappa \int_{\Stwo^2}
  (\tmmathbf{u} (\xi) \cdot \tmmathbf{n} (\xi))^2 \mathd \xi,
  \label{eq:micromagnefs2}
\end{equation}
where $\tmmathbf{n} (\xi) \assign \xi$ is the normal field to the unit
sphere and $\kappa \in \RR$ is an effective anisotropy parameter. Here, we have denoted by $\grad^{\ast} : H^1 ( \Stwo^2, \RR^3)
\rightarrow L^2 ( \Stwo^2, \RR^3)$ the tangential gradient on
$\Stwo^2$.

The existence of minimizers can be easily obtained using direct methods of the calculus of variations and  non-uniqueness of minimizers follows due to the invariance of the energy 
$\Fk$ under the orthogonal group.  An exact characterization of the minimizers in this problem is a non-trivial task and so far has been carried out only numerically \cite{sloika17}. However, sometimes it is enough to obtain a meaningful lower bound on the energy in order to gain some information of the ground states. This lower bound is typically obtained by relaxing the constraint $\tmmathbf{u} \in \Stwo^2$ to the following weaker constraint
\begin{equation}
  \frac{1}{4 \pi} \int_{\Stwo^2} | \tmmathbf{u} (\xi) |^2 \mathd \xi = 1.
  \label{eq:relaxedconstraint}
\end{equation}
This kind of relaxation, which physically corresponds to a passage from
classical physics to a probabilistic quantum mechanics perspective, has been
proved to be useful in obtaining non-trivial lower
bounds of the ground state micromagnetic energy (see eg {\cite{BrownA1968}}). 
Mathematically, replacing a constraint  $\tmmathbf{u} \in \Stwo^2$ with \eqref{eq:relaxedconstraint} puts us in a realm of Poincare-type inequalities, where in many cases the
relaxed problem can be solved exactly and the dependence of the
minimizers on the geometrical and physical properties of the model made
explicit. Sometimes this relaxation turns out to be helpful to obtain sufficient conditions for
minimizers to have specific geometric structures (see eg {\cite{BrownA1968}}).



We note that the constraint $| \tmmathbf{u} |^2 = 1$ {\tmabbr{a.e.}} on $\Stwo^2$ is
  equivalent to the following two energy constraints in terms of the $L^2$ and
  $L^4$ norms:
  \begin{equation}
    \frac{1}{4 \pi} \int_{\Stwo^2} | \tmmathbf{u} (\xi) |^2 \mathd \xi = 1
    \quad \text{and} \quad \frac{1}{4 \pi} \int_{\Stwo^2} | \tmmathbf{u} (\xi)
    |^4 \mathd \xi = 1.
  \end{equation}
 This observation follows from the Cauchy-Schwartz inequality
  \begin{equation}
    4 \pi = (| \tmmathbf{u} |^2, 1)_{L^2 ( \Stwo^2, \RR^3)}
    \leqslant \| | \tmmathbf{u} |^2 \|_{L^2 ( \Stwo^2, \RR^3)}  \|
    1 \|_{L^2 \left( \Stwo^2 \right)} = 4 \pi, \label{eq:condnecconst}
  \end{equation}
 where equality holds when $| \tmmathbf{u} |^2$ is a constant. Therefore our relaxed problem is the one obtained by removing the
  $L^4$ constraint.

\vskip 0.2cm

\noindent{\bf Main results.} Our results include the precise characterization of the minimal value and global minimizers of the energy
functional $\Fk$, defined in \eqref{eq:micromagnefs2}, on the space of $H^1 ( \Stwo^2, \RR^3)$ vector
fields satisfying the relaxed constraint \eqref{eq:relaxedconstraint}. In particular, we prove the following Poincar\'e-type inequality:

\begin{theorem}[Poincar{\'e} inequality on $\Stwo^2$]
  \label{thm:mainthmpoinc}Let $\kappa \in \RR$. For every $\tmmathbf{u} \in
  H^1 ( \Stwo^2, \RR^3)$ the following inequality holds:
  \begin{equation}
    \int_{\Stwo^2} \left| \grad^{\ast}_{\xi} \tmmathbf{u} (\xi) \right|^2
    \mathd \xi + \kappa \int_{\Stwo^2} (\tmmathbf{u} (\xi) \cdot \tmmathbf{n}
    (\xi))^2 \mathd \xi \; \geqslant \gamma (\kappa) \int_{\Stwo^2} |
    \tmmathbf{u} (\xi) |^2 \mathd \xi, \label{eq:PoincInequ}
  \end{equation}
  with 
\begin{equation} \label{eq:gamma}
\gamma (\kappa) \assign \left\{
\begin{array}{cc}
 \kappa+2 & \quad \text{if }\;\kappa \leqslant -4, \\
 \frac{1}{2} (\nobracket (\kappa + 6) -
  \sqrt{\kappa^2 + 4 \kappa + 36} \nobracket) & \quad \text{if }\;\kappa > - 4 \, .
\end{array}
\right.
\end{equation}
For any $\kappa \in \RR$ the equality in
  {\tmem{{\eqref{eq:PoincInequ}}}} holds {\tmstrong{if and only if}} the function $\tmmathbf{u}$ has
    the following form in terms of vector
  spherical harmonics $($see Section~\emph{\ref{sec:setup}}, Definition~\emph{\ref{def:sh}}$)$
  \begin{equation} \label{eq:globminimizer}
    \tmmathbf{u} (\xi) = c_0 \tmmathbf{y}_{0, 0}^{(1)} (\xi) + \sum_{j = - 1}^1 \sigma_j \tmmathbf{y}_{1, j}^{(1)} (\xi) + \tau_j
    \tmmathbf{y}_{1, j}^{(2)} (\xi),
  \end{equation}
  where coefficients $c_0, (\sigma_j, \tau_j)_{|j|\leqslant 1}$ are defined as follows
\begin{itemize}
\item  if $\kappa < - 4$ then $c_0 = \pm \sqrt{4\pi}$, $\sigma_j=\tau_j=0$ for $|j|\leqslant 1$;
\item if $\kappa > - 4$ then 
\begin{equation} \label{eq:u1andu2relmainthm}
c_0=0,\quad \tau_j = \frac{- 2 \sqrt{2}}{(\gamma(\kappa) - 2)} \sigma_j \quad \forall | j | \leqslant 1,\quad \sum_{|j|\leqslant 1} \sigma_j^2 = 2 \pi \frac{- (\kappa + 2) + \sqrt{\kappa^2 + 4 \kappa
  + 36}}{\sqrt{\kappa^2 + 4 \kappa + 36}};
\end{equation}
    \item if $\kappa = - 4$ then
\begin{equation} \label{eq:u1andu2rel2mainthm}
\tau_j = \frac{ \sqrt{2}}{2} \sigma_j \quad \forall | j | \leqslant 1,\quad 2 c_0^2 + 3 \sum_{|j|\leqslant 1} \sigma_j^2 = 8
\pi .
\end{equation}
  \end{itemize}
\end{theorem}

\begin{figure}[t]
  \scalebox{0.95}{\includegraphics{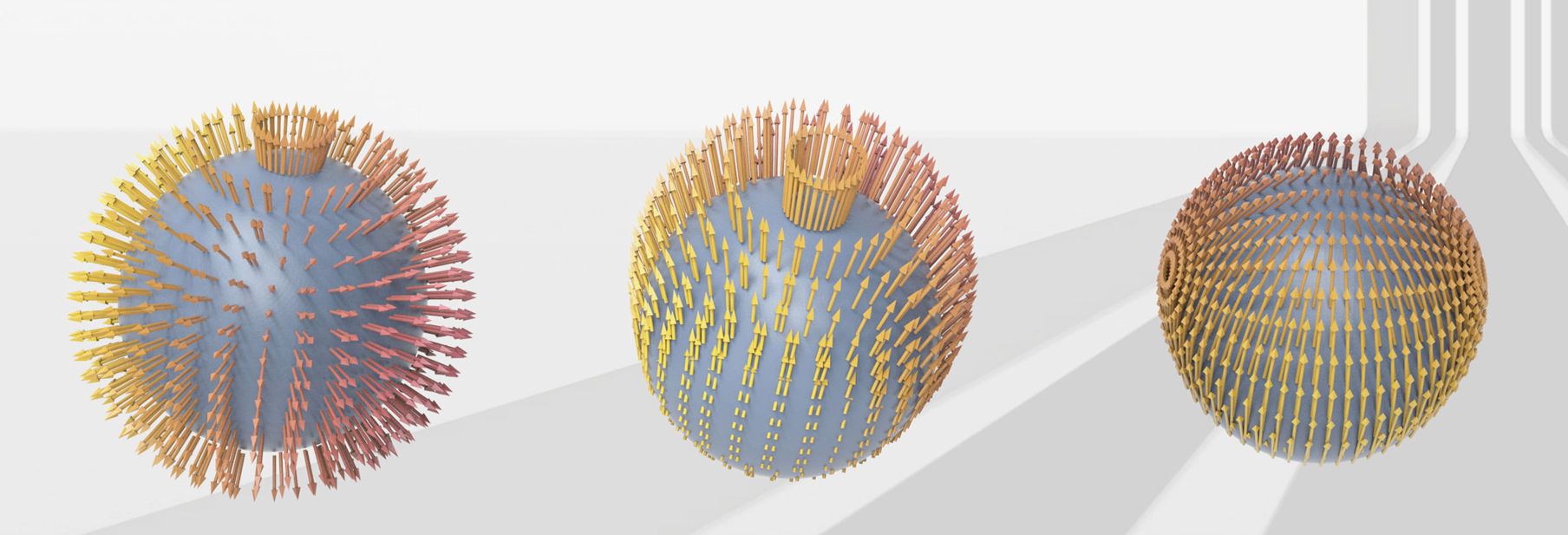}}
  \caption{\label{fig:3} Examples of vector fields
for which the equality sign is attained in the Poincar{\'e} inequality {\eqref{eq:PoincInequ}}. (Left) $\kappa = - 8$;
  (Center) $\kappa = - 4$; (Right) $\kappa = 6$.}
\end{figure}


We discover, surprisingly, that for $k \leqslant -4$ the unique minimizer of the relaxed problem coincides with the unique minimizer of $\Fk$ under {\bf the pointwise constraint }$| \tmmathbf{u} (\xi) |=1$. Thus, as a byproduct of Theorem~\ref{thm:mainthmpoinc} we obtain the following characterization of micromagnetic ground states in thin spherical shells.
\begin{theorem}[Micromagnetic ground states in thin spherical shells] \label{thm:2MMGSs}
For every $\kappa \in \RR$, the normal vector
fields $\pm \tmmathbf{n} (\xi)$ are stationary points of the micromagnetic
energy functional $\Fk$ given by {\eqref{eq:micromagnefs2}} on the space $H^1(\Stwo^2, \Stwo^2)$. Moreover, they are strict local
minimizers for every $\kappa < 0$ and are unstable for $\kappa >0$. If $\kappa \leqslant - 4$, the normal
vector fields $\pm \tmmathbf{n} (\xi)$ are the only global minimizers of
$\Fk$.
\end{theorem}

\begin{remark} Although the inequality  \eqref{eq:PoincInequ} holds for any $\kappa\in\RR$, it is sometimes more convenient to restate it in the standard form where both the term on the right side and the term on the left side are non-negative. Therefore when $\kappa \geqslant 0$ we can use {\tmem{{\eqref{eq:PoincInequ}}}} and if $\kappa<0$ we note that 
  $| \tmmathbf{u} (\xi) \times \tmmathbf{n} (\xi) |^2 - |
  \tmmathbf{u} (\xi) |^2 = - (\tmmathbf{u} (\xi) \cdot \tmmathbf{n} (\xi))^2$,
  and rewrite relation {\tmem{{\eqref{eq:PoincInequ}}}}  in the following way
  \begin{equation}\label{eq:PoincInequ-}
    \int_{\Stwo^2} \left| \grad^{\ast}_{\xi} \tmmathbf{u} (\xi) \right|^2
    \mathd \xi + | \kappa | \int_{\Stwo^2} | \tmmathbf{u} (\xi) \times
    \tmmathbf{n} (\xi) |^2 \mathd \xi \; \geqslant \; (| \kappa | + \gamma
    (\kappa)) \int_{\Stwo^2} | \tmmathbf{u} (\xi) |^2 \mathd \xi,
  \end{equation}
  with $| \kappa | \geqslant | \kappa | + \gamma (\kappa) \geqslant 0$ and the
  tangential part of the vector field appearing on the left-hand side.
  
Plots of the best constants $\kappa \in \RR \mapsto
\gamma (k)$ and $\kappa \in \RR \mapsto
\gamma (k) + |\kappa|$ for $\kappa>0$ and $\kappa<0$, respectively,  are given in Figure~\ref{Fig:bestconst}. Examples of vector fields
for which the equality sign is attained in {\eqref{eq:PoincInequ}} are
depicted in Figure~\ref{fig:3}. We note that for $\kappa < -4$ the minimizing
configurations are normal vector fields, for $\kappa \gg 1$ the tangential configurations are favoured and for the critical case $\kappa=-4$ various minimizing states may coexist.
\end{remark}

\begin{remark}\label{rem12}
  Note that the maximum value of $\gamma (\kappa)$ (see Figure~\ref{Fig:bestconst}) is reached at $\kappa=+ \infty$,
  where $\gamma (+ \infty) = 2$. It follows that for purely
  {\tmem{tangential}} vector fields one has the Poincar{\'e} inequality
  \begin{equation} \label{eq:PoincareTangent}
    \frac{1}{2} \int_{\Stwo^2} \left| \grad^{\ast}_{\xi} \tmmathbf{u} (\xi)
    \right|^2 \mathd \xi \; \geqslant \int_{\Stwo^2} | \tmmathbf{u} (\xi) |^2
    \mathd \xi .
  \end{equation}
The inequality \eqref{eq:PoincareTangent} is sharp as equality is achieved, for instance, by a vector field $\tmmathbf{u}(\xi)=\pm \sqrt{4\pi} \tmmathbf{y}_{1, 0}^{(2)} (\xi)$. In fact, one can characterize all vector fields delivering optimal Poincar{\'e} constant by taking the limit for $\kappa\to+\infty$ of the coefficients $\tau_j$ in \eqref{eq:u1andu2relmainthm}. 
\end{remark}

\begin{remark}
We note that Theorem~\ref{thm:2MMGSs} implies that the minimizers of micromagnetic energy don't have full radial symmetry in the case $\kappa>0$. It follows from the fact that the only radially symmetric vector fields are $\pm \tmmathbf{n} (\xi)$ and these are unstable for $\kappa>0$. 
\end{remark}

\begin{figure}[t]
\includegraphics[width=5.2in]{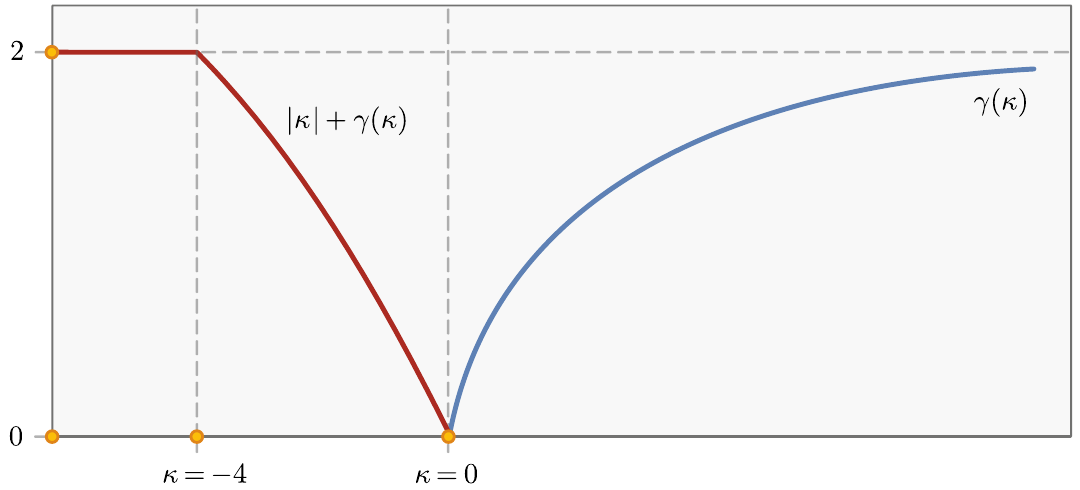}
  \caption{\label{Fig:bestconst} The values of the best constants in the Poincar{\'e} inequalities \eqref{eq:PoincInequ} and \eqref{eq:PoincInequ-} for $\kappa>0$ and $\kappa<0$, respectively.}
\end{figure}

\begin{remark}
It is worth noting that,  in the language of modern physics, the two ground states $\pm \tmmathbf{n}$ carry a different {\it skyrmion number} (or  topological charge). Indeed, since $\text{deg} (\pm \tmmathbf{n})=\pm 1$, by Hopf theorem~{\cite{milnor1997topology}}, these two configurations cannot be homotopically mapped one into the other and are, therefore, topologically protected against external perturbations and thermal fluctuations.
These considerations make the two ground states $\pm \tmmathbf{n}$ promising in view of novel spintronic devices \cite{fert2008nobel,fert2013skyrmions}.

We also want to point out a correspondence between our Theorem~\ref{thm:2MMGSs} and Brown's fundamental theorem on fine ferromagnetic particles~{\cite{BrownA1968,di2012generalization,alouges2015liouville}}, as Theorem~\ref{thm:2MMGSs}  implies an existence of a critical value $\kappa_0<0$ below which the only ground states are $\pm \tmmathbf{n} (\xi)$.
\end{remark}

%

\vskip 0.2cm
In the following, in Section~\ref{sec:setup}, we define suitable vector spherical harmonics. Afterwards, in Section~\ref{sec:representationsequences}, by the means of these vector spherical harmonics, we recast the minimization problem for $\Fk$ as a constrained minimization problem on a suitable space of sequences. Then, in Section~\ref{sec:uppboundmin}, by proper use of the Euler-Lagrange equations in sequence space, we derive necessary minimality conditions which allow us to reduce the infinite dimensional problem to a finite dimensional one. Finally, arguments based on the method of Lagrange multipliers complete the proof of Theorem~\ref{thm:mainthmpoinc} and afterwards of Theorem~\ref{thm:2MMGSs}.
\section{Notation and setup. Vector Spherical Harmonics}\label{sec:setup}
 In this section, we define a natural basis  and characterize vector spherical harmonics on the unit sphere $\Stwo^2$, see {\cite{freeden2008spherical}}. 
Every point $\xi \in \Stwo^2$ can be expressed via the polar coordinates
parametrization 
\begin{equation}
\sigma (\varphi, t) = \left( {\sqrt{1 - t^2}} \cos \,
\varphi,  {\sqrt{1 - t^2}} \sin \, \varphi, t \right),
\end{equation}
where
$\varphi \in [0, 2 \pi)$ is the longitude, $t = \cos \theta \in [- 1, 1]$ is
the polar distance and $\theta \in [0, \pi]$ the latitude.

We can define the surface gradient operator $\grad_{\xi}^{\ast}$ for {\tmabbr{a.e.}} $\xi \in \Stwo^2$ in the following way
\begin{equation}
  \grad_\xi^{\ast} = \varepsilon^{\varphi} \frac{1}{{\sqrt{1 - t^2}} }
  \partial_{\varphi} + \varepsilon^t {\sqrt{1 - t^2}} \partial_t,
\end{equation}
where $\varepsilon^{\varphi} (\varphi, t) \assign (- \sin \varphi, \cos
\varphi, 0)$, $\varepsilon^t (\varphi, t) \assign \left( - t \cos \varphi, - t
\sin \varphi, {\sqrt{1 - t^2}}  \right)$. For any $u
\in C^2 ( \Stwo^2, \RR )$, the Laplace-Beltrami operator
is defined as 
\begin{equation}
\Delta_{\xi}^{\ast} u (\xi) \assign \grad_{\xi}^{\ast} \cdot
\grad_{\xi}^{\ast} u (\xi). 
\end{equation}

\begin{notation}
  We denote by $\NN$ the set of positive integers, by $\NN_0$ the set of
  non-negative integers. For every $n \in \NN$ we set $\NN_n \assign \{
  1, 2, \ldots, n \}$ and $\ZZ_n \assign \{ 0, \pm 1, \ldots, \pm n \}$, for every $N \in \NN_0$ we introduce the set $\myJ_N \subseteq \NN_0 \times \ZZ$ consisting of all
  pairs $(n, j) \in \NN_0 \times \ZZ$ such that $n \leqslant N$ and $| j |
  \leqslant n$. We set $\myJ \assign \myJ_{\infty}$.
\end{notation}

Vector spherical harmonics are an extension of the scalar spherical harmonics
to square-integrable vector fields on the sphere; in fact, they can be
introduced in terms of the scalar spherical harmonics and their derivatives.
Motivated by different physical problems, various sets of vector spherical
harmonics have been introduced in the literature. The system that best fit our
purposes is the one introduced in {\cite{barrera1985vector}}, and obtained
from the splitting of vector fields into a radial and tangential component.
We have the following
definition (see {\cite{freeden2008spherical}}).

\begin{definition}\label{def:sh}
  The vector spherical harmonics $\tmmathbf{y}_{n, j}^{(1)}, \tmmathbf{y}_{n,
  j}^{(2)}$, and $\tmmathbf{y}_{n, j}^{(3)}$ of degree $n$ and order $j$, with
  $(n, j) \in \myJ$, are defined by
  \begin{equation}
    \tmmathbf{y}_{n, j}^{(1)} (\xi) \assign   Y_{n, j} (\xi)\tmmathbf{n}(\xi), \quad
    \tmmathbf{y}_{n, j}^{(2)} (\xi) \assign \frac{1}{\sqrt{\ns}}
    \grad^{\ast}_{\xi} Y_{n, j} (\xi), \quad \tmmathbf{y}_{n, j}^{(3)} (\xi)
    \assign \frac{1}{\sqrt{\ns}} \grad^{\ast}_{\xi} \wedge Y_{n, j} (\xi) ,
    \label{eq:VSHSys}
  \end{equation} where $\ns \assign n (n + 1)$.
  Here, for every $(n, j) \in \myJ$, the function $Y_{n, j}$ is the
  real-valued scalar spherical harmonics of degree $n$ and order $j$, defined
  by
  \begin{equation}
    Y_{n, j} (\xi) \assign \left\{ \begin{array}{ll}
      \sqrt{2} X_{n, | j |} (t) \cos (j \varphi) & \text{if } - n \leqslant j
      < 0,\\
      X_{n, 0} (\theta) & \text{if } j = 0,\\
      \sqrt{2} X_{n, j} (t) \sin (j \varphi) & \text{if } 0 < j \leqslant n,
    \end{array} \right.
  \end{equation}
  where for every $t \in [- 1, 1]$ and every $0 \leqslant j \leqslant n$
  \begin{eqnarray}
    X_{n, j} (t) & = & (- 1)^j \sqrt{\left( \frac{2 n + 1}{4 \pi} \right)
    \frac{(n - j) !}{(n + j) !}} P_{n, j} (t), 
  \end{eqnarray}
  and $P_{n, j}$ is the associate Legendre polynomial given by $P_{n, j} (t)
  \assign \frac{1}{2^n n!} (1 - t^2)^{j / 2} \partial^{n + j}_t (t^2 - 1)^n$.
\end{definition}

It is well-known ({\tmabbr{cf.}}
{\cite{barrera1985vector,nedelec2001acoustic}}) that the system $(Y_{n,
j})_{(n, j) \in \myJ}$ so defined is a complete orthonormal system for $L^2
\left( \Stwo^2, \RR \right)$, consisting of eigenfunctions of the Laplace-Beltrami
operator. Precisely, for every $n \in \NN_0$ we have $- \lapl_{\xi}^{\ast}
Y_{n, j} = \ns Y_{n, j}$ with $\ns \assign n (n + 1)$. Not so widely known
seems to be that the system of vector spherical harmonics 
is complete in $L^2 ( \Stwo^2, \RR^3)$ and forms an
orthonormal system ({\tmabbr{cf.}}~{\cite{freeden2008spherical}}). Therefore,
any vector field $\tmmathbf{u} \in L^2 ( \Stwo^2, \RR^3 )$ can be
represented by its Fourier series:
\begin{equation}
  \sum_{i \in \NN_3} \sum_{(n, j) \in \myJ} \hat{u}^{(i)} (n, j)
  \tmmathbf{y}_{n, j}^{(i)} = \tmmathbf{u} \quad \text{in $L^2 (
  \Stwo^2, \RR^3 )$} \, , \label{eq:FSSPHARM}
\end{equation}
with the Fourier coefficients $\hat{u}^{(i)}$ being given by $\hat{u}^{(i)}
(n, j) \assign (\nobracket \tmmathbf{u}, \tmmathbf{y}_{n, j}^{(i)}
\nobracket)_{L^2 ( \Stwo^2, \RR^3)}$.

 As the minimizers of our problem will be fully characterized in terms of the
  first vector spherical harmonics, it is worth to explicitly write down their explicit expressions.
  By the relation $\tmmathbf{y}_{n,
  j}^{(1)} (\xi) \assign Y_{n, j}(\xi)\tmmathbf{n}(\xi)$ we get, for $n = 0$, that
\begin{equation}
\tmmathbf{y}_{0, 0}^{(1)} (\xi) = \frac{1}{\sqrt{4 \pi}} \tmmathbf{n}(\xi).
\end{equation}
For $n = 1$, we get
\begin{eqnarray}
  \mu^{(1)} \tmmathbf{y}_{1, - 1}^{(1)} (\xi) & = & \sin \theta \cos \varphi
  \, \tmmathbf{n} (\xi), \\
  \mu^{(1)} \tmmathbf{y}_{1, 0}^{(1)} (\xi) & = & \cos \theta \,\tmmathbf{n}
  (\xi), \\
  \mu^{(1)} \tmmathbf{y}_{1, 1}^{(1)} (\xi) & = & \sin \theta \sin \varphi
 \, \tmmathbf{n} (\xi), 
\end{eqnarray}
with $\mu^{(1)} \assign \sqrt{4 \pi / 3}$. Also, by the relation
$\tmmathbf{y}_{n, j}^{(2)} (\xi) \assign \frac{1}{\sqrt{n_{\ast}}}
\nabla^{\ast}_{\xi} Y_{n, j} (\xi)$, we obtain, for $n = 1$, the following
identities:
\begin{eqnarray}
  \mu^{(2)} \tmmathbf{y}_{1, - 1}^{(2)} (\xi) & = & \cos \theta \cos \varphi
 \, \tmmathbf{\tau}_{\theta} (\xi) - \sin \varphi \, \tmmathbf{\tau}_{\varphi}
  (\xi), \\
  \mu^{(2)} \tmmathbf{y}_{1, 0}^{(2)} (\xi) & = & - \sin \theta
  \,\tmmathbf{\tau}_{\theta} (\xi), \\
  \mu^{(2)} \tmmathbf{y}_{1, 1}^{(2)} (\xi) & = & \cos \theta \sin \varphi
  \,\tmmathbf{\tau}_{\theta} (\xi) + \cos \varphi\, \tmmathbf{\tau}_{\varphi}
  (\xi), 
\end{eqnarray}
with $\mu^{(2)} \assign \sqrt{8 \pi / 3}$, $\tmmathbf{\tau}_{\theta} (\xi)
\assign (\cos \theta \cos \varphi, \cos \theta \sin \varphi, - \sin \theta)$,
and $\tmmathbf{\tau}_{\varphi} (\xi) \assign (- \sin \varphi, \cos \varphi,
0)$. Note that the tangent vectors $\tmmathbf{\tau}_{\theta}$ and
$\tmmathbf{\tau}_{\varphi}$ have unit norms.
 The previous expressions turn out to be extremely useful to obtain both a qualitative and a quantitative comprehension of the energy landscape as in Figure~\ref{fig:3}.
\begin{remark}
  Throughout the paper, we use summations which formally involve also $\tmmathbf{y}_{0,
  0}^{(2)} =\tmmathbf{y}_{0, 0}^{(3)} = 0$, with the understanding that $\hat{u}^{(2)} (0, 0) =
  \hat{u}^{(3)} (0, 0) = 0$. Indeed, although 
  these vectors are not officially present in the orthonormal system of vector
  spherical harmonics, such a convention
  allows us to express the Fourier series representation of $\tmmathbf{u}$ in
  the compact form $\sum_{i \in \NN_3} \sum_{(n, j) \in \myJ} \hat{u}^{(i)}
  (n, j) \tmmathbf{y}_{n, j}^{(i)} $.

\end{remark}

\section{Representation of the energy in a space of sequences}\label{sec:representationsequences}

In this section we are going to rewrite the energy \eqref{eq:micromagnefs2} in terms of sequences using Fourier representation \eqref{eq:FSSPHARM}.
According to the representation formula {\eqref{eq:FSSPHARM}}, every
vector field $\tmmathbf{u} \in H^1 ( \Stwo^2, \RR^3)$ can be
expressed in the form
\begin{equation}
  \tmmathbf{u}= \sum_{i \in \NN_3} \sum_{(n, j) \in \myJ} \hat{u}_i (n, j)
  \tmmathbf{y}_{n, j}^{(i)} \quad \text{in } L^2 ( \Stwo^2, \RR^3),
\end{equation}
with the Fourier coefficients $\hat{u}^{(i)}$ being given by $\hat{u}^{(i)}
(n, j) \assign (\nobracket \tmmathbf{u}, \tmmathbf{y}_{n, j}^{(i)}
\nobracket)_{L^2 ( \Stwo^2, \RR^3)}$.
Also, if $\tmmathbf{u}$ is a smooth vector field, we have $\|
\grad^{\ast}_{\xi} \tmmathbf{u} \|^2_{L^2 ( \Stwo^2, \RR^3)}
= (- \Delta^{\ast}_{\xi} \tmmathbf{u}, \tmmathbf{u})_{L^2 ( \Stwo^2,
\RR^3 )}$. Hence, by making use of the relations ({\tmabbr{cf.}}
{\cite[p.237]{freeden2008spherical}})
\begin{eqnarray}
  - \Delta^{\ast} \tmmathbf{y}^{(1)}_{n, j} & = & \left( \ns + 2 \right)
  \tmmathbf{y}^{(1)}_{n, j} - 2 \sqrt{\ns} \tmmathbf{y}^{(2)}_{n, j}, \\
  - \Delta^{\ast} \tmmathbf{y}^{(2)}_{n, j} & = & \ns
  \tmmathbf{y}^{(2)}_{n, j} - 2 \sqrt{\ns} \tmmathbf{y}^{(1)}_{n, j}, \\
  - \Delta^{\ast} \tmmathbf{y}^{(3)}_{n, j} & = & \ns
  \tmmathbf{y}^{(3)}_{n, j}, 
\end{eqnarray}
where $\ns \assign n (n + 1)$, we infer that for every $\tmmathbf{u} \in
C^{\infty} \left( \Stwo^2, \RR^3 \right)$
\begin{eqnarray}
  - \Delta^{\ast}_{\xi} \tmmathbf{u} (\xi) & = & \sum_{(n, j) \in \myJ}
  \hat{u}_1 (- \Delta^{\ast}_{\xi} \tmmathbf{y}^{(1)}) + \hat{u}_2 (-
  \Delta^{\ast}_{\xi} \tmmathbf{y}^{(2)}) + \hat{u}_3 (- \Delta^{\ast}_{\xi}
  \tmmathbf{y}^{(3)}) \\
  & = & \sum_{(n, j) \in \myJ} \left( \left( \ns + 2 \right) \hat{u}_1 - 2
  \sqrt{\ns} \hat{u}_2 \right) \tmmathbf{y}^{(1)}_{n, j} + \left( \ns
  \hat{u}_2 - 2 \sqrt{\ns} \hat{u}_1 \right) \tmmathbf{y}^{(2)}_{n, j} + \ns
  \hat{u}_3 \tmmathbf{y}^{(3)}_{n, j} . 
\end{eqnarray}
with the understanding that $\hat{u}_2 (0, 0) = \hat{u}_3 (0, 0) = 0$ and
$\hat{u}_1 = \hat{u}_1 (n, j)$, $\hat{u}_2 = \hat{u}_2 (n, j)$, and $\hat{u}_3
= \hat{u}_3 (n, j)$. Thus, for every $\tmmathbf{u} \in C^{\infty} \left(
\Stwo^2, \RR^3 \right)$,
\begin{equation}
  \int_{\Stwo^2} \left| \grad^{\ast}_{\xi} \tmmathbf{u} (\xi) \right|^2 \mathd
  \xi = \sum_{(n, j) \in \myJ} \left( \ns + 2 \right) \hat{u}_1^2 - 4
  \sqrt{\ns} \hat{u}_1  \hat{u}_2 + \ns \hat{u}_2^2 + \ns \hat{u}_3^2,
\end{equation}
and, by density, the same relation holds for every $\tmmathbf{u} \in H^1
\left( \Stwo^2, \RR^3 \right)$. Also, a straightforward calculation shows that
\begin{equation}
  \int_{\Stwo^2} (\tmmathbf{u} (\xi) \cdot \tmmathbf{n} (\xi))^2 \mathd \xi
  = \sum_{(n, j) \in \myJ} \hat{u}_1^2 (n, j) .
\end{equation}
Therefore, the surface energy {\eqref{eq:micromagnefs2}}, in the
sequence space, reads as the functional
\begin{equation}
  \Gk (\hat{\tmmathbf{u}}) = \sum_{(n, j) \in \myJ} \left( \ns - 2 + \kappa
  \right) \hat{u}_1^2 + \left( 2 \hat{u}_1 - \sqrt{\ns} \hat{u}_2 \right)^2 +
  \ns \hat{u}_3^2 . \label{eq:EnergyinVSH}
\end{equation}
Denoting by $\ltwoJ$ the classical Hilbert space of square-summable sequences
endowed with the inner product $\langle \hat{\tmmathbf{u}}, \hat{\tmmathbf{v}}
\rangle \assign \sum_{(n, j) \in \myJ} \hat{u}_1 \hat{v}_1 + \hat{u}_2
\hat{v}_2 + \hat{u}_3 \hat{v}_3$, the natural domain of $\Gk$ is the subspace
$\vhs$ of $\ltwoJ$ consisting of those sequences in $\hat{\tmmathbf{u}} \in
\ltwoJ$ such that $\sqrt{\ns} \hat{\tmmathbf{u}} \in \ltwoJ$. In $\vhs$ the
constraint {\eqref{eq:relaxedconstraint}} reads as
\begin{equation}
  \langle \hat{\tmmathbf{u}}, \hat{\tmmathbf{u}} \rangle \, = \sum_{(n, j) \in
  \myJ} \hat{u}_1^2 + \hat{u}_2^2 + \hat{u}_3^2 = \int_{\Stwo^2} |
  \tmmathbf{u} (\xi) |^2 \mathd \xi = 4 \pi \, .
  \label{eq:ELu20constraint4p}
\end{equation}
As before, in the previous relations, to shorten notation, we avoided to
explicitly write the dependence of $\hat{u}_1, \hat{u}_2, \hat{u}_3$ from $(j,
n)$.

\section{Proof of the Poincar{\'e} inequality (Theorem \ref{thm:mainthmpoinc})}\label{sec:uppboundmin}

In this section, we are going to prove the main result of this note -- Theorem~\ref{thm:mainthmpoinc}.
Without loss of
generality, we will focus on the case $\kappa \neq 0$, because for $\kappa =
0$ the only minimizers are the constant vector fields with unit modulus. Instead of 
working with the original continuous formulation \eqref{eq:micromagnefs2},
we introduce the equivalent formulation in terms of sequences:
\begin{equation}
  \min_{\hat{\tmmathbf{u}} \in \vhs} \Gk (\hat{\tmmathbf{u}}), \quad
  \text{subject to} \quad \frac{1}{4 \pi} \| \hat{\tmmathbf{u}}
  \|^2_{\mathbf{{\ell}}_2 \left( \myJ \right)} = 1\, , \label{problem:min}
\end{equation}
and provide a complete characterization of the minimizers of {\eqref{problem:min}}.

We split the proof into several steps and firstly prove the following useful lemma.

\begin{lemma}
  \label{lemma:3}For any $\kappa \in \RR$, the following upper bound on the energy \eqref{eq:EnergyinVSH} holds
  \begin{equation} 
    \min \Gk (\hat{\tmmathbf{u}}) \leqslant \min \left\{ 2 \pi \left( (\kappa
    + 6) - \sqrt{\kappa^2 + 4 \kappa + 36} \right), 4 \pi (2 + \kappa)
    \right\} < 8 \pi \label{eq:necconminimality} .
  \end{equation}
  Moreover, if $\hat{\tmmathbf{u}} = (\hat{u}_1, \hat{u}_2, \hat{u}_3) \in
  \vhs$ is a minimizer for $\Gk$ then: 
  \begin{enumerate}[(i)]
  \item[\tmstrong{i{\tmem{)}}}] The coefficients $\hat{u}_3 (n, j) = 0$ for any $(n, j)
  \in \myJ$. 
  \item[\tmstrong{ii{\tmem{)}}}] If $\Gk (\hat{\tmmathbf{u}}) < 4 \pi
  (2 + \kappa)$ then $\hat{u}_1 (0, 0) = 0$. 
  \item[\tmstrong{iii{\tmem{)}}}] The coefficients 
  $\hat{\tmmathbf{u}} (n, j) = 0$ for any $n \geqslant 2$ and all $| j |
  \leqslant n$.  
  \end{enumerate}
  
\end{lemma}

\begin{proof}
  We provide a simple test function $\hat{\tmmathbf{u}}_* (n, j)$ by setting all its terms to $0$ except
  $\hat{u}_1
  (1, 1)$ and $\hat{u}_2 (1, 1)$. Therefore the minimum value of $\Gk$ is 
  less than the minimum of $\alpha_{\kappa} (x, y) = (\kappa + 4) x^2 - 4
  \sqrt{2} x_{} y + 2 y^2$ under constraint $x^2+y^2 = 4\pi$. By studying the minima
  of $(\alpha_{\kappa} \circ \gamma) (t)$ with $\gamma (t) = \sqrt{4 \pi} 
  (\cos t, \sin t)$, it is easily seen that
  \begin{equation}
    \min_{(x,y) \in \sqrt{4 \pi}  \Stwo^1} \alpha_{\kappa} (x,y) = 2 \pi
    \left( (\kappa + 6) - \sqrt{\kappa^2 + 4 \kappa + 36} \right) .
  \end{equation}
  Note that, $\kappa^2 + 4 \kappa + 36 > 0$ for every $\kappa \in \RR$ and
  moreover $2 \pi \left( (\kappa + 6) - \sqrt{\kappa^2 + 4 \kappa + 36}
  \right) < 8 \pi$ for every $\kappa \in \RR$, therefore
  \begin{equation}
    \min \Gk (\hat{\tmmathbf{u}}) \leqslant 2 \pi \left( (\kappa + 6) -
    \sqrt{\kappa^2 + 4 \kappa + 36} \right) < 8 \pi \quad \forall \kappa \in
    \RR .
  \end{equation}
  Next, we provide another test function $\hat{\tmmathbf{u}}^* (n, j)$ by setting all its terms to $0$ 
except $\hat{u}_1 (0, 0)$ and 
  $\hat{u}_1 (1, 1)$. Therefore the minimum of $\Gk$ is  less than the
  minimum of $\beta_{\kappa} (x, y) = (\kappa + 2) x^2 + (\kappa + 4) y^2$ on
  $\sqrt{4 \pi}  \Stwo^1$. By studying the minima of $(\beta_{\kappa} \circ
  \gamma) (t)$ with $\gamma (t) = \sqrt{4 \pi}  (\cos t, \sin t)$, it is
  easily seen that
  \begin{equation}
    \min_{\sigma \in \sqrt{4 \pi}  \Stwo^1} \beta_{\kappa} (\sigma) = 4 \pi (2
    + \kappa) .
  \end{equation}
  Therefore, for every $\kappa \in \RR$, relation
  {\eqref{eq:necconminimality}} holds.
  
  \
  
  {\noindent}{\tmstrong{{\tmem{i}})}} We compute the first variation of
  $\Gk$ around the generic point $\hat{\tmmathbf{u}} \in \vhs$ to obtain the following 
  Euler-Lagrange equations
  \begin{equation}
    \sum_{(j, n) \in \myJ} \left( \ns + 2 + \kappa \right) \hat{u}_1 \hat{v}_1
    - 2 \sqrt{\ns} (\hat{u}_1 \hat{v}_2 + \hat{v}_1 \hat{u}_2) + \ns
    (\hat{u}_2 \hat{v}_2 + \hat{u}_3 \hat{v}_3) = \lambda
    (\hat{\tmmathbf{u}}) \cdot \langle \hat{\tmmathbf{u}}, \hat{\tmmathbf{v}}
    \rangle,
  \end{equation}
  with $\lambda (\hat{\tmmathbf{u}}) \in \RR$ the Lagrange multiplier coming
  from the constraint {\eqref{eq:ELu20constraint4p}}. Plugging
  $\hat{\tmmathbf{v}} \assign \hat{\tmmathbf{u}}$ and taking into account
  {\eqref{eq:ELu20constraint4p}}, we obtain $\lambda (\hat{\tmmathbf{u}}) =
  \frac{1}{4 \pi} \Gk (\hat{\tmmathbf{u}})$. Thus, the Euler Lagrange equation
  reads as
  \begin{equation}
    \frac{1}{4 \pi} \Gk (\hat{\tmmathbf{u}}) \langle \hat{\tmmathbf{u}},
    \hat{\tmmathbf{v}} \rangle = \sum_{(j, n) \in \myJ} \left( \ns + 2 +
    \kappa \right) \hat{u}_1 \hat{v}_1 - 2 \sqrt{\ns} (\hat{u}_1 \hat{v}_2 +
    \hat{v}_1 \hat{u}_2) + \ns (\hat{u}_2 \hat{v}_2 + \hat{u}_3 \hat{v}_3) .
    \label{eq:ELOriginal}
  \end{equation}
  for every $\hat{\tmmathbf{v}} \in \vhs$.
  
  We test {\eqref{eq:ELOriginal}} against the sequence $\hat{\tmmathbf{v}}
  \assign (\hat{v}_1, \hat{v}_2, \hat{v}_3)$ with $\hat{v}_1 = \hat{v}_2 = 0$
  and $\hat{v}_3 = \hat{e}_{(n, j)}$, with $\hat{e}_{n, j}$ denoting the
  sequence $(n', j') \in \myJ \mapsto \hat{e}_{n, j} (n', j') \in \RR$ such
  that $\hat{e}_{n, j} (n, j) = 1$ and $\hat{e}_{n, j} (n', j') = 0$ if $(n',
  j') \neq (n, j)$. We get that
  \begin{equation}
    \frac{1}{4 \pi} \Gk (\hat{\tmmathbf{u}}) \hat{u}_3 (n, j) = \ns
    \hat{u}_3 (n, j)  \label{eq:testel1}
  \end{equation}
  for any $n \geqslant 1$ and any $| j | \leqslant n$. Thus, for $n \geqslant
  1$ we have $\Gk (\hat{\tmmathbf{u}})  = 4 \pi \ns \geqslant 8 \pi$
  whenever $\hat{u}_3 (n, j) \neq 0$. Since the minimum of energy is strictly less then $8 \pi$ we necessarily have $\hat{u}_3 (n, j) = 0$ for any $n
  \geqslant 1$. This proves the assertion.
  
  \
  
  {\noindent}{\tmstrong{{\tmem{ii}})}} We now evaluate
  {\eqref{eq:ELOriginal}} on $\hat{\tmmathbf{v}} \assign (\hat{v}_1,
  \hat{v}_2, \hat{v}_3)$, first with $\hat{v}_2 = \hat{v}_3 = 0$ and
  $\hat{v}_1 = \hat{e}_{(n, j)}$, then on $\hat{v}_2 = \hat{e}_{(n,
  j)}$, $\hat{v}_3 = 0$ and $\hat{v}_1 = 0$. We get the following two
  relations
  \begin{eqnarray}
    \frac{1}{4 \pi} \Gk (\hat{\tmmathbf{u}}) \hat{u}_1 (n, j) & = & \left(
    \ns + 2 + \kappa \right) \hat{u}_1 (n, j) - 2 \sqrt{\ns} \hat{u}_2 (n, j) 
    \label{eq:testel2}\\
    \frac{1}{4 \pi} \Gk (\hat{\tmmathbf{u}}) \hat{u}_2 (n, j) & = & - 2
    \sqrt{\ns} \hat{u}_1 (n, j) + \ns \hat{u}_2 (n, j)  \label{eq:testel3}
  \end{eqnarray}
  For $n = 0$, relation {\eqref{eq:testel2}} gives $\Gk (\hat{\tmmathbf{u}})
  \hat{u}_1 (0, 0) = 4 \pi (2 + \kappa) \hat{u}_1 (0, 0)$ so that if
  $\hat{\tmmathbf{u}}$ is a minimizer and $\hat{u}_1 (0, 0) \neq 0$, the
  minimum energy agrees with the limiting value $4 \pi (2 + \kappa)$. 
  Therefore, if the
  minimal energy is strictly less than $4 \pi (2 + \kappa)$, then necessarily
  $\hat{u}_1 (0, 0) = 0$. This proves the statement.
  
  \
  
  {\noindent}{\tmstrong{{\tmem{iii}})}} If
  $\hat{\tmmathbf{u}}$ is a minimizer of $\Gk$ then for $n \geqslant 1$, using {\eqref{eq:testel2}}
  and {\eqref{eq:testel3}}, we have that $\hat{u}_1(n, j) = 0$ if and only if  $\hat{u}_2(n, j) = 0$.
  Equivalently, for any $n \geqslant 1$, $\hat{u}_1 (n, j) \hat{u}_2 (n, j) =
  0$ implies $\hat{u}_1 (n, j) = 0$ and $\hat{u}_2 (n, j) = 0$.
  
  We now focus on the indices $n \geqslant 1$ and, using above observation, rewrite 
  relations {\eqref{eq:testel2}} and {\eqref{eq:testel3}} into the form
  \begin{eqnarray}
    \frac{1}{4 \pi} \Gk (\hat{\tmmathbf{u}}) \hat{u}_1 (n, j) \hat{u}_2 (n, j)
    & = & \left( \ns + 2 + \kappa \right) \hat{u}_1 (n, j) \hat{u}_2 (n, j)
    - 2 \sqrt{\ns} \hat{u}_2^2 (n, j)  \label{eq:testel2new}\\
    \frac{1}{4 \pi} \Gk (\hat{\tmmathbf{u}}) \hat{u}_2 (n, j) \hat{u}_1 (n, j)
    & = & - 2 \sqrt{\ns} \hat{u}_1^2 (n, j) + \ns \hat{u}_2 (n, j)
    \hat{u}_1 (n, j) .  \label{eq:testel3new}
  \end{eqnarray}
  If for some $n \geqslant 1$ the product $\hat{u}_1 (n, j) \hat{u}_2 (n, j)$
  is negative then from {\eqref{eq:testel2new}} and {\eqref{eq:testel3new}} we
  get
  \begin{eqnarray}
    \Gk (\hat{\tmmathbf{u}}) & = & 4 \pi \left[ \left( \ns + 2 + \kappa
    \right) - 2 \sqrt{\ns} \frac{\hat{u}_2^2 (n, j)}{\hat{u}_1 (n, j)
    \hat{u}_2 (n, j)} \right] > 4 \pi (\kappa + 2) \\
    \Gk (\hat{\tmmathbf{u}}) & = & 4 \pi \left[ \ns - 2 \sqrt{\ns}
    \frac{\hat{u}_1^2 (n, j)}{\hat{u}_1 (n, j) \hat{u}_2 (n, j)} \right] > 8
    \pi 
  \end{eqnarray}
  and $\hat{\tmmathbf{u}}$ is not a minimizer as a consequence of
  {\eqref{eq:necconminimality}}. Thus, if $\hat{\tmmathbf{u}}$ is a minimizer
  of $\Gk$ then
  \begin{equation}
    \tmop{sign} (\hat{u}_1 (n, j)) = \tmop{sign} (\hat{u}_2 (n, j)) \quad
    \text{for any } n \geqslant 1. \label{eq:exlemma:5}
  \end{equation}
  Hence, from {\eqref{eq:testel2}} and {\eqref{eq:testel3}} we infer
  \begin{eqnarray}
    \Gk (\hat{\tmmathbf{u}}) & = & 4 \pi \left[ \left( \ns + 2 + \kappa
    \right) - 2 \sqrt{\ns} \frac{| \hat{u}_2 (n, j) |}{| \hat{u}_1 (n, j) |}
    \right],  \label{eq:testel2sign}\\
    \Gk (\hat{\tmmathbf{u}}) & = & 4 \pi \left[ \ns - 2 \sqrt{\ns} \frac{|
    \hat{u}_1 (n, j) |}{| \hat{u}_2 (n, j) |} \right] . 
    \label{eq:testel3sign}
  \end{eqnarray}
  Imposing the condition $\Gk (\hat{\tmmathbf{u}}) \leqslant 4 \pi (\kappa +
  2)$ in {\eqref{eq:testel2sign}} and the condition $\Gk (\hat{\tmmathbf{u}})
  < 8 \pi$ in {\eqref{eq:testel3sign}} we get that if $\hat{\tmmathbf{u}}$ is
  a minimizer then necessarily $\left( \ns - 2 \right) | \hat{u}_2 (n, j) | <
  4 | \hat{u}_2 (n, j) |$, but this cannot be the case for $n \geqslant 2$.
  Therefore, necessarily $\hat{u}_1 (n, j) = \hat{u}_2 (n, j) = 0$ for any $n
  \geqslant 2$. This concludes the proof.
\end{proof}

Combining the results stated in Lemma \ref{lemma:3}, we can reduce the
infinite dimensional minimization problem for $\Gk$ to a finite dimensional
one. Precisely, we have the following proposition.

\begin{proposition}
  \label{thm:main}The minimization problem for $\Gk$, subject to the
  constraint {\tmem{{\eqref{eq:ELu20constraint4p}}}}, reduces to the
  minimization, in the variables $\sigma \assign (\hat{u}_1 (0, 0), \hat{u}_1 (1, j),
  \hat{u}_2 (1, j))_{| j | \leqslant 1}$, of the constrained function $g_{\kappa} : \sqrt{4 \pi} \Stwo^6 \rightarrow
  \RR^+$ given by
  \begin{equation}
    g_{\kappa} (\sigma) = (\kappa + 2) \hat{u}_1^2 (0, 0) + \sum_{j = - 1}^1
    \kappa \hat{u}_1^2 (1, j) + \left( 2 \hat{u}_1 (1, j) - \sqrt{2} \hat{u}_2
    (1, j) \right)^2 . \label{finitedimproblemS6}
  \end{equation}
  Precisely, any minimizer $\hat{\tmmathbf{u}}_{\star} = (\hat{u}_1 (n, j),
  \hat{u}_2 (n, j), \hat{u}_3 (n, j))_{(n, j) \in \myJ}$ of $\Gk$ has all the terms zero except for those presented in
  $\sigma$, and coming fom minimizing $g_{\kappa}$.  
 Specifically, the following complete characterization of the energy landscape holds:
  \begin{itemize}
\item  If $\kappa < - 4$, the minimum value of the energy is given by
$\mathcal{G}_{\kappa} (\hat{\tmmathbf{u}}_{\star}) = 4 \pi (\kappa + 2)$ and,
in this case, $\hat{u}_1 (0, 0)$ is the only non-zero variable. Therefore,
necessarily $\hat{u}_1 (0, 0) = \pm \sqrt{4 \pi}$.
    
    \item If $\kappa > - 4$ the minimum value of the energy is given by $\Gk
(\hat{\tmmathbf{u}}_{\star}) = 4 \pi \gamma_+(\kappa)$ with  $\gamma_+(\kappa):=\frac{1}{2} \left( (\kappa + 6) - \sqrt{\kappa^2 + 4
    \kappa + 36} \right)$. In this case, necessarily $\hat{u}_1(0, 0) = 0$ and
\begin{equation} \label{eq:u1andu2rel}
\hat{u}_2 (1, j) = \frac{- 2 \sqrt{2}}{(\gamma_+(\kappa) - 2)} \hat{u}_1 (1,j) \quad \forall | j | \leqslant 1.
\end{equation}
The minimum
value is reached on any vector $\hat{\sigma} = (\hat{u}_1 (1, j))_{| j |
\leqslant 1}$ such that
\begin{equation}
  | \hat{\sigma} |^2 = 2 \pi \frac{- (\kappa + 2) + \sqrt{\kappa^2 + 4 \kappa
  + 36}}{\sqrt{\kappa^2 + 4 \kappa + 36}} . \label{eq:configinfinite}
\end{equation}
    \item If $\kappa = - 4$, the minimum value of the energy is given by
$\mathcal{G}_{\kappa} (\hat{\tmmathbf{u}}_{\star}) = - 8
\pi$ and it is reached on any vector $\sigma$ such that \eqref{eq:u1andu2rel} holds and $2 \hat{u}_1^2 (0, 0) + 3 | \hat{\sigma} |^2 = 8
\pi$.
  \end{itemize}
\end{proposition}
\begin{remark} \label{rmk:samermkofPoinc}
  The limiting value $\kappa = - 4$ represents a special case in which different topological states may coexist. Indeed, for $|
  \hat{\sigma} | = 0$ we recover the solutions $\hat{u}_1 (0, 0) \assign \pm
  \sqrt{4 \pi}$ formally arising as the limit for $\kappa \rightarrow - 4^-$
  of the family of minimization problems for $g_{\kappa}$. Similarly, for $\hat{u}_1 (0, 0) = 0$, we recover the minimal solutions arising as the limit for $\kappa \rightarrow - 4^+$ of the family of minimization problems for $g_{\kappa}$.
\end{remark}

\begin{proof}
  According to Lemma \ref{lemma:3}, the Euler-Lagrange equations
  {\eqref{eq:ELOriginal}}, can be simplified to read, for every $\hat{\tmmathbf{v}} \in \vhs$, as
  \begin{equation}
    \frac{1}{4 \pi} \Gk (\hat{\tmmathbf{u}}) \langle \hat{\tmmathbf{u}},
    \hat{\tmmathbf{v}} \rangle = \sum_{(n, j) \in \myJ_1} \left( \ns + 2 +
    \kappa \right) \hat{u}_1 \hat{v}_1 - 2 \sqrt{\ns} (\hat{u}_1 \hat{v}_2 +
    \hat{v}_1 \hat{u}_2) + \ns (\hat{u}_2 \hat{v}_2 + \hat{u}_3 \hat{v}_3).
  \end{equation}
Taking, in the order,
  $\hat{\tmmathbf{v}} = (\hat{e}_{0, 0}, 0, 0)$, $\hat{\tmmathbf{v}} =
  (\hat{e}_{1, j}, 0, 0)$, $\hat{\tmmathbf{v}} = (0, \hat{e}_{1, j}, 0)$, we
  get that if $\hat{\tmmathbf{u}}$ is a minimizer, then
  \begin{eqnarray}
    \frac{1}{4 \pi} \Gk (\hat{\tmmathbf{u}}) \hat{u}_1 (0, 0) & = & (2 +
    \kappa) \hat{u}_1 (0, 0),  \label{eq:eq1}\\
    \frac{1}{4 \pi} \Gk (\hat{\tmmathbf{u}}) \hat{u}_1 (1, j) & = & (4 +
    \kappa) \hat{u}_1 (1, j) - 2 \sqrt{2} \hat{u}_2 (1, j),  \label{eq:eq2}\\
    \frac{1}{4 \pi} \Gk (\hat{\tmmathbf{u}}) \hat{u}_2 (1, j) & = & - 2
    \sqrt{2} \hat{u}_1 (1, j) + 2 \hat{u}_2 (1, j) .  \label{eq:eq3}
  \end{eqnarray}
  From equation {\eqref{eq:eq1}} and Lemma~\ref{lemma:3} we immediately obtain that $\hat{u}_1 (0, 0)
  \neq 0$ if, and only if, $\Gk (\hat{\tmmathbf{u}}) = 4 \pi (2 + \kappa) .$
  On the other hand, from {\eqref{eq:eq3}}, setting $G_{\kappa} \assign
  \frac{1}{4 \pi} \Gk (\hat{\tmmathbf{u}})$ and noting that $G_{\kappa} <
  2$, we obtain
  \begin{equation}
    \hat{u}_2 (1, j) = \frac{- 2 \sqrt{2}}{(G_{\kappa} - 2)} \hat{u}_1 (1,
    j) . \label{eq:tempu1equalu2}
  \end{equation}
Substituting this last expression into {\eqref{eq:eq2}} we obtain
  $(G_{\kappa} - 2) (G_{\kappa} - (4 + \kappa)) \hat{u}_1 (1, j) = 8
  \hat{u}_1 (1, j)$, and this, together with {\eqref{eq:tempu1equalu2}},
  implies that if $\hat{u}_1 (1, j) \neq 0$ for some $| j | \leqslant 1$, then
  $\hat{u}_2 (1, j)$ is different from zero too, and $(G - (4 + \kappa)) (G -
  2) = 8$, that is
  \begin{equation}
    \Gk (\hat{\tmmathbf{u}}) =4\pi \gamma_+(\kappa), \quad  \gamma_+(\kappa):=\frac{1}{2} \left( (\kappa + 6) - \sqrt{\kappa^2 + 4
    \kappa + 36} \right).
  \end{equation}
We have proved the following implication:
\[ \left(\exists | j | \leqslant 1 \quad  \hat{u}_1 (1, j) \neq 0\quad \text{{\tmstrong{or}}} \quad \hat{u}_2 (1, j) \neq 0 \right)
  \quad \Longrightarrow \quad  \mathcal{G}_{\kappa} (\hat{\tmmathbf{u}}) = 4\pi \gamma_+(\kappa) . \]
Therefore, if $\mathcal{G}_{\kappa} (\hat{\tmmathbf{u}}) \neq 4\pi \gamma_+(\kappa)$ then necessarily
\[  \hat{u}_1 (1, j) = \hat{u}_2 (1, j) = 0\quad \forall | j | \leqslant 1 . \]
Since $\gamma_+(\kappa) \leqslant  (\kappa +
2)$ if, and only if, $\kappa \geqslant - 4$, by \eqref{eq:necconminimality} we infer that for $\kappa < -
4$ we have $\mathcal{G}_{\kappa} (\hat{\tmmathbf{u}}) < 4 \pi \gamma_+(\kappa)$ and $\hat{u}_1 (1, j)
= \hat{u}_2 (1, j) = 0$ $\forall | j | \leqslant 1$. Since the variables in
$\sigma$ 
must be in $\sqrt{4 \pi} \Stwo^6$ this means that
$\hat{u}_1 (0, 0)$ is the only variable different from zero, and therefore necessarily equal to $\pm \sqrt{4 \pi}$.

On the other hand, from equation {\eqref{eq:eq1}} we immediately obtain that \tmcolor{red}{if}
$\hat{u}_1 (0, 0) \neq 0$ then $\mathcal{G}_{\kappa} (\hat{\tmmathbf{u}}) = 4
\pi (2 + \kappa)$, which, in turn, implies $\kappa \leqslant - 4$. 
Therefore, if $\kappa > - 4$ then necessarily $\hat{u}_1 (0, 0) = 0$ and, due
to the constraint, at least one of the $\hat{u}_1 (1, j)$ is different from
zero. Thus, $G_{\kappa} \assign \frac{1}{4 \pi} \mathcal{G}_{\kappa}
(\hat{\tmmathbf{u}}) = \gamma_+ (\kappa)$. This observation, in combination
with {\eqref{eq:tempu1equalu2}}, implies that for $\kappa > - 4$ the problem
trivialize to the minimization of
\begin{equation}
  g_{\kappa} (\hat{\sigma}) = \left( \frac{\kappa (\gamma_+ (\kappa) - 2)^2 +
  4 \gamma^2_+ (\kappa)}{(\gamma_+ (\kappa) - 2)^2} \right) | \hat{\sigma}
  |^2, \quad \hat{\sigma} \assign (\hat{u}_1 (1, j))_{| j | \leqslant 1},
\end{equation}
subject to the constraint $| \hat{\sigma} |^2 
= 4 \pi (\gamma_+ (\kappa)
- 2)^2 / ((\gamma_+ (\kappa) - 2)^2 + 8)$. This leads to the already computed
minimal value $g_{\kappa} (\hat{\sigma}) = \gamma_+ (\kappa)$ reached on any
vector $\hat{\sigma} = (\hat{u}_1 (1, j))_{| j | \leqslant 1}$ such that
{\eqref{eq:configinfinite}} holds.

Finally, for $\kappa = - 4$, we have $\gamma_+ (- 4) = - 2$, and again by
{\eqref{eq:tempu1equalu2}}, the problem trivialize to the minimization of
\begin{equation}
  g_{\kappa} (\sigma) = - 2 \hat{u}_1^2 (0, 0) - 3 | \hat{\sigma} |^2, \quad
  \sigma \assign (\hat{u}_1 (0, 0), \hat{\sigma}),
\end{equation}
subject to the constraint $2 \hat{u}_1^2 (0, 0) + 3 | \hat{\sigma} |^2 = 8
\pi$. This leads to the minimal value $g_{\kappa} (\sigma) = - 8 \pi$ reached
on any vector $\sigma \assign (\hat{u}_1 (0, 0), \hat{\sigma})$ such that $2
\hat{u}_1^2 (0, 0) + 3 | \hat{\sigma} |^2 = 8 \pi$.
\end{proof}

\noindent{\bf Finalizing the proof of Theorem~\ref{thm:mainthmpoinc}.} Going back
to the minimization problem {\eqref{eq:micromagnefs2}}, \eqref{eq:relaxedconstraint} for the energy
functional $\Fk$, the results of Proposition~\ref{thm:main} immediately
translate into the context of Theorem~\ref{thm:mainthmpoinc} via the Fourier
isomorphism that maps $\Fk$ into $\Gk$. It is therefore sufficient to apply
the results to $\Fk (\tilde{\tmmathbf{u}})$ with $\tilde{\tmmathbf{u}} \assign
\sqrt{4 \pi} \tmmathbf{u}/ \| \tmmathbf{u} \|_{L^2 \left( \Stwo^2, \RR^3
\right)}$.

\smallskip

\noindent{\bf Proof of Theorem~\ref{thm:2MMGSs}.}
  Due to the saturation constraint $| \tmmathbf{u}(\xi) |^2 = 1$ for a.e. $\xi\in\Stwo^2$, the
  Euler-Lagrange equations for $\Fk$ reads, in strong form, as
  \begin{equation}
    \tmmathbf{u} (\xi) \times (- \Delta_{\xi}^{\ast} \tmmathbf{u} (\xi) +
    \kappa (\tmmathbf{u} (\xi) \cdot \tmmathbf{n}(\xi)) \tmmathbf{n}(\xi))
    = \; 0
    \quad \forall \xi \in \Stwo^2 . \label{eq:ELFkappa}
  \end{equation}
  Since $- \Delta_{\xi}^{\ast} \tmmathbf{n} (\xi) = 2\tmmathbf{n} (\xi)$, the
  vector fields $\tmmathbf{u}_{\pm} (\xi) \assign \pm \tmmathbf{n} (\xi)$ 
  satisfy {\eqref{eq:ELFkappa}} and, therefore, are stationary points of
  $\Fk$. 
  
  Next, consider the second order
  variation $\Fk'' (\tmmathbf{u}, \cdot)$ of $\Fk$ at $\tmmathbf{u} \in H^1
  (\Stwo^2, \Stwo^2)$, which reads, for every $\tmmathbf{v} \in H^1 (\Stwo^2,
 \RR^3)$ such that $\tmmathbf{u} (\xi) \cdot
  \tmmathbf{v} (\xi) = 0$ for {\tmabbr{a.e.}} in $\Stwo^2$, as
  \begin{equation}
    \Fk'' (\tmmathbf{u}, \tmmathbf{v}) = \int_{\Stwo^2} | \nabla^{\ast}_{\xi}
    \tmmathbf{v} |^2 - | \nabla^{\ast}_{\xi} \tmmathbf{u} |^2 | \tmmathbf{v}
    |^2 \mathd \xi + \kappa \int_{\Stwo^2} (\tmmathbf{v} \cdot \tmmathbf{n})^2 -
    (\tmmathbf{u} \cdot \tmmathbf{n})^2 | \tmmathbf{v} |^2 \mathd \xi .
  \end{equation}
  In particular, for $\tmmathbf{u} (\xi) \assign \pm \tmmathbf{n} (\xi)$,
  noting that $| \nabla^{\ast}_{\xi} \tmmathbf{n} (\xi) |^2 = 2$, we get
  \begin{eqnarray}\label{eqSVar}
    \Fk'' (\pm \tmmathbf{n}, \tmmathbf{v}) & = & \int_{\Stwo^2} |
    \nabla^{\ast}_{\xi} \tmmathbf{v} |^2 - (\kappa + 2) | \tmmathbf{v} |^2
    \mathd \xi . 
  \end{eqnarray}
  Now, for $\tmmathbf{u} (\xi) \assign \pm \tmmathbf{n} (\xi)$, the condition
  $\tmmathbf{u} (\xi) \cdot \tmmathbf{v} (\xi) = 0$ {\tmabbr{a.e.}} in $\Stwo^2$
  forces the variation $\tmmathbf{v}$ to be tangent to $\Stwo^2$. Thus, the Poincar{\'e} inequality {\eqref{eq:PoincareTangent}}
  holds and we end up with the estimate
  \[ \Fk'' (\pm \tmmathbf{n}, \tmmathbf{v}) \geqslant - \kappa \int_{\Stwo^2} |
     \tmmathbf{v} |^2 \mathd \xi, \]
  from which the strict local minimality follows.
  
  To show instability of  $\tmmathbf{u} (\xi) \assign \pm \tmmathbf{n} (\xi)$ for $\kappa>0$ we return to the second variation \eqref{eqSVar}. Using a test function $\tmmathbf{u}(\xi)=\sqrt{4\pi} \tmmathbf{y}_{1, 0}^{(2)} (\xi)$ from the Remark~\ref{rem12} we obtain negativity of the second variation which implies instability of $\tmmathbf{u} (\xi) \assign \pm \tmmathbf{n} (\xi)$.
  
  Finally, for $\kappa \leqslant - 4$, the global minimality of $\pm
  \tmmathbf{n} (\xi)$ is clear from Theorem \ref{thm:mainthmpoinc} and the
  fact that $\Fk$ is constrained to $H^1 (\Stwo^2, \Stwo^2)$.
  
\section{Acknowledgements}
GDF acknowledges support from the Austrian Science Fund (FWF)
through the special research program {\tmem{Taming complexity in partial
differential systems}} (Grant SFB F65) and of the Vienna Science and
Technology Fund (WWTF) through the research project {\tmem{Thermally
controlled magnetization dynamics}} (Grant MA14-44), VS acknowledges support from 
EPSRC grant EP/K02390X/1 and Leverhulme grant RPG-2014-226.

The work AZ is supported by the Basque Government through the BERC 2018-2021
program, by Spanish Ministry of Economy and Competitiveness MINECO through BCAM
Severo Ochoa excellence accreditation SEV-2017-0718 and through project MTM2017-82184-R
funded by (AEI/FEDER, UE) and acronym ``DESFLU''.
\par The authors would like to thank the Isaac Newton Institute for Mathematical Sciences for support and hospitality during the programme {\it ``The design of new materials"} when work on this paper was undertaken. This work was supported by: EPSRC grant numbers EP/K032208/1 and EP/R014604/1.

\end{document}